\documentclass{amsart}

\newtheorem{theorem}{Theorem}[section]

\theoremstyle{definition}

\newtheorem{proposition}[theorem]{Proposition}
\newtheorem{corollary}[theorem]{Corollary}
\newtheorem{remark}[theorem]{Remark}

\theoremstyle{remark}

\newcommand{\be}{\begin{equation}}
\newcommand{\ee}{\end{equation}}

\numberwithin{equation}{section}

\begin{document}

\title{On an algebraic formula and applications to group action on manifolds}

\author{Ping Li}
\address{Department of Mathematics, Tongji University, Shanghai 200092, China}
\email{pingli@tongji.edu.cn}
\thanks{The first author was partially supported by National Natural Science Foundation of China (Grant No. 11101308)}

\author{Kefeng Liu}
\address{Department of Mathematics, University of California at Los Angeles, Los Angeles, CA 90095, USA
and Center of Mathematical Science, Zhejiang University, Hangzhou
310027, China}
\email{liu@math.ucla.edu}





\begin{abstract}
In this paper we consider a purely algebraic result. Then given a
circle or cyclic group of prime order action on a manifold, we will
use it to estimate the lower bound of the number of fixed points. We
also give an obstruction to the existence of $\mathbb{Z}_p$ action
on manifolds with isolated fixed points when $p$ is a prime.
\end{abstract}

\maketitle

\section{Introduction}
In this short note we first prove a purely algebraic result, which
stems from the localization formulae of group actions on manifolds,
and the idea of which has been used in several literatures
(\cite{PT}, \cite{LL}, \cite{Li}, \cite{LT}, \cite{CKP}). Then,
using this result, we shall give some applications to circle action
and finite cyclic group action on unitary manifolds and smooth
manifolds. The rest of this section is to introduce and prove this
algebraic result and the next two sections is devoted to
applications.

Let $F$ be a field (finite or infinite). What we are mainly
concerned with is the real number filed $\mathbb{R}$ and finite
field $\mathbb{Z}/p\mathbb{Z}=:\mathbb{Z}_p$. Here $\mathbb{Z}$ is
the integer ring and $p$ is a prime.

Suppose we have a set of $r$ elements $\{P_1,\cdots,P_r\}$. We call
such a set a \emph{weighted set over $F$} if, for each $P_i$, $1\leq
i\leq r$, we associate $n+1$ numbers
$\mu_i,a_1^{(i)},\cdots,a_n^{(i)}$ in $F$ to it. $\mu_i$ is called
the \emph{coefficient} of $P_i$ and $a_1^{(i)},\cdots,a_n^{(i)}$ are
called \emph{characteristic numbers} of $P_i$. We will see in the
next section that such a system appears naturally as the fixed
points of group action on manifolds and the coefficient and
characteristic numbers of these elements will be induced from the
representation on the tangent spaces of the corresponding fixed
points.

Now we will introduce the concept of \emph{stable} weighted set.

Let $\lambda=(1^{m_1(\lambda)}2^{m_2(\lambda)}\cdots
n^{m_n(\lambda)})$ be a partition of weight $w$, i.e.,
$m_j(\lambda)$ are all non-negative integers and
$\sum_{j=1}^{n}j\cdot m_j(\lambda)=w$. For convenience we set
$$a_{\lambda}^{(i)}:=\prod_{j=1}^{n}(a_j^{(i)})^{m_j(\lambda)}.$$
We call a weighted set $\{P_1,\cdots,P_r\}$ over $F$ \emph{stable}
if for any $0\leq w<n$ and any partition
$\lambda=(1^{m_1(\lambda)}2^{m_2(\lambda)}\cdots n^{m_n(\lambda)})$
of weight $w$, we have
\be\label{stable}\Gamma(\lambda):=\sum_{i=1}^{r}\mu_i\cdot
a_{\lambda}^{(i)}=0.\ee The reason of this definition lies in the
localization formulae of group actions and its underlying geometric
meaning will be clear in the next section.

For a partition $\lambda=(1^{m_1(\lambda)}2^{m_2(\lambda)}\cdots
n^{m_n(\lambda)})$, we define
$$m(\lambda):=\textrm{max}\{m_1(\lambda),m_2(\lambda),\ldots,m_n(\lambda)\}.$$
With these notations understood, we can state the main result in
this section.

\begin{proposition}\label{algebraic result}
Let $\{P_1,\cdots,P_r\}$ be a stable weighted set over a field $F$.
If there exists a partition
$\lambda=(1^{m_1(\lambda)}2^{m_2(\lambda)}\cdots n^{m_n(\lambda)})$
of weight $n$ such that $m(\lambda)\geq r$, then
$$\Gamma(\lambda)=\sum_{i=1}^{r}\mu_i\cdot a_{\lambda}^{(i)}=0.$$
\end{proposition}
\begin{proof}
Without loss of generality, we may assume $m_1(\lambda)=m(\lambda)$.
Let \be\label{set}\{a_1^{(i)}~|~1\leq i\leq
r\}=\{s_1,\ldots,s_l\}\in F\ee By definition $s_1,\ldots,s_l$ are
mutually distinct, $l\leq r$, and $l=r$ if and only if
$a_1^{(i)},\ldots,a_n^{(i)}$ are mutually distinct. Define
$$A_{t}:=\sum_{\substack{1\leq i\leq r \\ a_1^{(i)}=s_{t}}}\mu_i\cdot(a_2^{(i)})^{m_2(\lambda)}
(a_3^{(i)})^{m_3(\lambda)}\cdots(a_n^{(i)})^{m_n(\lambda)},\qquad
1\leq t\leq l.$$ Now let us consider the following $m_1(\lambda)$
partitions:
$$\lambda^{(j)}:=(1^j2^{m_2(\lambda)}3^{m_3(\lambda)}\cdots n^{m_n(\lambda)}),\qquad 0\leq j\leq m_1(\lambda)-1.$$
The weights of these $\lambda^{(j)}$ are all less than $n$ as the
weight of $\lambda$ is exactly $n$. Then following (\ref{stable}) we
have
\begin{eqnarray}\label{matrix}
\left\{\begin{array}{l}
A_{1}+A_{2}+\cdots+A_{l}=0\\
s_{1}A_{1}+s_{2}A_{2}+\cdots+s_{l}A_{l}=0\\
\vdots\\
(s_{1})^{m_1(\lambda)-1}A_{1}+(s_{2})^{m_1(\lambda)-1}A_{2}+\cdots+(s_{l})^{m_1(\lambda)-1}A_{l}=0
\end{array}
\right. \end{eqnarray}

Note that $l\leq r$ and by assumption $m_1(\lambda)=m(\lambda)\geq
r$, which follows $l\leq m_1(\lambda)$. By definition
$s_1,\cdots,s_l$ are mutually distinct, which means the determinant
of the coefficient matrix of the first $l$ lines of (\ref{matrix})
is \be\left|\begin{array}{cccc}
1& 1 & \cdots & 1\\
s_1 & s_2 & \cdots & s_l\\
\vdots & \vdots & \ddots &\vdots\\
(s_1)^{l-1} & (s_2)^{l-1} & \cdots &(s_l)^{l-1}
\end{array}\right|=\prod_{1\leq i<j\leq l}(s_j-s_i)\neq 0,\nonumber\ee
  which is a nonsingular Vandermonde matrix.

  Thus
$$A_1=A_2=\cdots=A_l=0.$$
Therefore, \be \begin{split}
\Gamma(\lambda)&=\sum_{i=1}^{r}\mu_i\cdot
a_{\lambda}^{(i)}\\
&=\sum_{i=1}^{r}
(a_{1}^{(i)})^{m_1(\lambda)}\cdot\mu_i\cdot(a_2^{(i)})^{m_2(\lambda)}
\cdots(a_n^{(i)})^{m_n(\lambda)}\\
&=\sum_{t=1}^{l}(s_t)^{m_1(\lambda)}\cdot A_t=0
\end{split}\nonumber\ee
\end{proof}
\begin{remark}
The idea of this proof was first used by Pelayo-Tolman (\cite{PT},
Lemma 8), then by the present authors (\cite{LL}, Lemmas 3.1 and
3.2) and the first author (\cite{Li}, Lemmas 3.2 and 3.2).
L\"{u}-Tan used this idea in their proof of Theorem 1.1 in
\cite{LT}, which was extracted by Cho-Kim-Park in (\cite{CKP},
Theorem 2.2).
\end{remark}
\begin{corollary}
Let $\{P_1,\cdots,P_r\}$ be a stable weighted set over a field $F$.
If we define
$$m:=\textrm{max}\{m(\lambda)~|~\lambda:~\textrm{partitions of weight}~n~\textrm{such that}~\Gamma(\lambda)\neq 0\},$$
then $$r\geq m+1.$$
\end{corollary}
Note that if in Proposition \ref{algebraic result} $F$ is a finite
field, then by (\ref{set}) we know $l\leq |F|$, where $|F|$ is the
cardinality of $F$. Therefore we have the following implication from
the process of proving Proposition \ref{algebraic result}.

\begin{corollary}
Let $\{P_1,\cdots,P_r\}$ be a stable weighted set over a finite
field $F$. If there exists a partition
$\lambda=(1^{m_1(\lambda)}2^{m_2(\lambda)}\cdots n^{m_n(\lambda)})$
of weight $n$ such that $m(\lambda)\geq |F|$, then
$\Gamma(\lambda)=0$.
\end{corollary}

\begin{corollary}\label{coro}
Let $\{P_1,\cdots,P_r\}$ be a stable weighted set over a finite
field $F$. If we define
$$m:=\textrm{max}\{m(\lambda)~|~\lambda:~\textrm{partitions of weight}~n~\textrm{such that}~\Gamma(\lambda)\neq 0\},$$
then $$|F|\geq m+1.$$
\end{corollary}

\section{Applications to circle actions}
For our purpose, here we only review Bott residue formula for
$4n$-dimensional smooth manifolds. For general case we refer to
(\cite{LL}, Section 2).

Let $N^{4n}$ be a $4n$-dimensional connected, oriented smooth
manifold with a (smooth) circle action whose fixed points are
non-empty and isolated, say $\{P_1,P_2\cdots,P_r\}$. At each fixed
point $P_i$, the tangent space $T_{P_i}N$ splits as a real
$S^1$-module induced from the isotropy representation as follows
$$T_{p_{i}}N=\bigoplus_{j=1}^{2n}V^{(i)}_{j},$$
where each $V_{j}^{(i)}$ is a real $2$-dimensional plane. We choose
an isomorphism of complex plane $\mathbb{C}$ with $V_{j}^{(i)}$
relative to which the representation of $S^{1}$ on $V_{j}^{(i)}$ is
given by $e^{\sqrt{-1}\theta}\mapsto e^{\sqrt{-1}k_{j}^{(i)}\theta}$
with $k_{j}^{(i)}\in\mathbb{Z}-\{0\}$. We can assume the rotation
numbers $k^{(i)}_{1},\cdots,k^{(i)}_{2n}$ be chosen in such a way
that the usual orientations on the summands
$V^{(i)}_{j}\cong\mathbb{C}$ induce the given orientation on
$T_{p_{i}}N$. Note that these $k^{(i)}_{1},\cdots,k^{(i)}_{2n}$ are
uniquely defined up to even number of sign changes. In particular,
their product $\prod_{j=1}^{2n}k_{j}^{(i)}$ is well-defined.

Let $p_j\in H^{4j}(M;\mathbb{Z})$ $(1\leq j\leq n)$ be the $j$-th
Pontrjagin class of $N^{4n}$ and
$\lambda=(1^{m_1(\lambda)}2^{m_2(\lambda)}\cdots n^{m_n(\lambda)})$
be a partition. Then we can define the corresponding Pontrjagin
number $p_{\lambda}[N]$ as follows
$$p_{\lambda}[N]:=<p_1^{m_1(\lambda)}p_2^{m_2(\lambda)}\cdots p_n^{m_n(\lambda)},[N]>.$$
Here $[N]$ is the fundamental class of $N^{4n}$ determined by the
orientation and $<\cdot,\cdot>$ is the Kronecker pairing. By
definition $p_{\lambda}[N]=0$ unless the weight of $\lambda$ is $n$.

Let $e_j(x_1,\cdots,x_{2n})$ be the $j$-th elementary symmetric
polynomial in the variables $x_1,\cdots,x_{2n}$. At each fixed point
$P_i$, we can associate to $n+1$ numbers
$\mu_i,a_1^{(i)},\cdots,a_n^{(i)}$ as follows.
$$\mu_i=\frac{1}{\prod_{j=1}^{2n}k^{(i)}_j},\qquad a_j^{(i)}=e_j\big((k_1^{(i)})^2,\cdots,(k_{2n}^{(i)})^2\big),\qquad 1\leq j\leq n.$$
The following result is a special case of Bott residue formula
(\cite{AS}, p.598).
\begin{theorem}[Bott residue formula]
With all above notations understood and suppose
$\lambda=(1^{m_1(\lambda)}2^{m_2(\lambda)}\cdots n^{m_n(\lambda)})$
is any partition whose weight is no more than $n$, then we have
$$\sum_{i=1}^{r}\mu_i\cdot a_{\lambda}^{(i)}=p_{\lambda}[N].$$
\end{theorem} Combining this formula with Proposition \ref{algebraic result}
will lead to the following result, which is parallel to (\cite{CKP},
Theorem 1.3) and a generalization of (\cite{LL}, Theorem 1.4) in the
smooth case.

\begin{theorem}
Suppose $N^{4n}$ is a $4n$-dimensional connected, closed and
oriented smooth manifold and there exists a partition
$\lambda=(1^{m_1(\lambda)}2^{m_2(\lambda)}\cdots n^{m_n(\lambda)})$
of weight $n$ such that $p_{\lambda}[N]\neq 0$, then any circle
action on $N^{4n}$ has at least $m(\lambda)+1$ fixed points.
\end{theorem}

\begin{remark}\label{remark}
In the statement of above theorem, we require no restriction that
the circle action must have non-empty isolated fixed points. The
reason is as follows. The non-vanishing of some Pontrjagin number
guarantees that the fixed points of any circle action must be
non-empty. If the fixed-point set is not isolated, then at least one
connected component of it is a submanifold of positive dimension and
in this case there are infinitely many fixed points.
\end{remark}

\begin{corollary}
Suppose $N^{4n}$ is a $4n$-dimensional connected, closed and
oriented smooth manifold. Let
$$m:=\textrm{max}\{m(\lambda)~|~p_{\lambda}[N]\neq 0\}.$$
Then any circle action on $N^{4n}$ has at least $m+1$ fixed points.
\end{corollary}

\section{Applications to $\mathbb{Z}_p$ actions}
Let us recall the notation of unitary manifold (some people call it
stably almost-complex manifold or weakly almost-complex manifold),
which is a generalization of that of almost-complex manifold.

Let $M^{2n}$ be a $2n$-dimensional connected, closed and smooth
manifold. $M^{2n}$ is called a \emph{unitary manifold} if it is
endowed with a complex vector bundle structure on the stable tangent
bundle. More precisely, we can give a complex vector bundle
structure on $TM\oplus\theta^{2l}$, where $TM$ is the tangent bundle
of $M$ and $\theta^{2l}$ is the product bundle
$M\times\mathbb{R}^{2l}$. $TM\oplus\theta^{2l}$ can be oriented from
the complex vector bundle structure and $\theta^{2l}$ is also
oriented in the usual way. These orientations induce an orientation
of $TM$. Hereafter a unitary manifold $M^{2n}$ is always oriented in
such a way.

Now suppose the unitary manifold $M^{2n}$ admits a $\mathbb{Z}_p$
action preserving the given complex vector bundle structure. Here of
course $p$ is a prime and $\mathbb{Z}_p$ is the cyclic group of
order $p$. Thus $\mathbb{Z}_p$ is a finite field. Moreover we
suppose this action has isolated fixed points, say $P_1,\cdots,P_r$.
Then at each fixed point $P_i$, $\mathbb{Z}_p$ acts linearly on the
complex vector space $V_i:=T_{P_i}M\oplus\mathbb{R}^{2l}$ and the
complex subspace of fixed vectors of this action is exactly
$\mathbb{R}^{2l}$. Thus $V_i/\mathbb{R}^{2l}$ is a complex
$\mathbb{Z}_p$-module without trivial factor, which is isomorphic to
$T_{P_i}M$ as real vector space. So we have $n$ weights
$k^{(i)}_1,\cdots,k^{(i)}_n\in\mathbb{Z}_p-\{0\}$ induced from this
$\mathbb{Z}_p$-representation.

Note that, at each $P_i$, the tangent space $T_{P_i}M$ has two
orientations. One is induced from that of $M$ and the other is
induced from the complex structure of $V_i/\mathbb{R}^{2l}$. We set
$\epsilon(P_i)=+1$ or $-1$ according to these two orientations
coincide or not. Clearly $\epsilon(P_i)=+1$ if $M^{2n}$ is an
almost-complex manifold.

Let $c_j\in H^{2j}(M;\mathbb{Z})$ $(1\leq j\leq n)$ be the $j$-th
Chern class of the complex vector bundle $TM\oplus\theta^{2l}$.
Given any partition $\lambda=(1^{m_1(\lambda)}2^{m_2(\lambda)}\cdots
n^{m_n(\lambda)})$, we can define the corresponding Chern number
$c_{\lambda}[M]$:
$$c_{\lambda}[M]:=<c_1^{m_1(\lambda)}c_2^{m_2(\lambda)}\cdots c_n^{m_n(\lambda)},[M]>.$$
By definition $c_{\lambda}[M]=0$ unless the weight of $\lambda$ is
$n$.

At each fixed point $P_i$ $(1\leq i\leq r)$, we can associate to
$n+1$ numbers $\mu_i,a_1^{(i)},\cdots,a_n^{(i)}\in\mathbb{Z}_p$ as
follows.
$$\mu_i=\epsilon(P_i)\cdot(\prod^{n}_{j=1}k^{(i)}_{j})^{-1},\qquad
a_j^{(i)}=e_j(k_1^{(i)},\cdots,k_n^{(i)}),\qquad 1\leq j\leq n.$$
Here $e_j(x_1,\cdots,x_n)$ is the $j$-th elementary symmetric
polynomial in the variables $x_1,\cdots,x_n$.

The following proposition is a special case of a result of
Kosniowski (\cite{Ko}, Theorem 1.1), which reduces the calculations
of the module $p$ Chern numbers of $M$ to the fixed points $\{P_i\}$
and is a $\mathbb{Z}_p$ analogue to Bott residue formula in the
circle action.

\begin{proposition}[Kosniowski]\label{Kos}
With all these above understood and let
$\lambda=(1^{m_1(\lambda)}2^{m_2(\lambda)}\cdots n^{m_n(\lambda)})$
be any partition whose weight is no more than $n$. Then we have
$$\sum_{i=1}^{r}\mu_i\cdot a_{\lambda}^{(i)}\equiv c_{\lambda}[M],\qquad \textrm{mod}~p.$$
\end{proposition}
Combining this proposition with Proposition \ref{algebraic result}
will lead to the following result.

\begin{theorem}\label{p-theorem}
Given a $2n$-dimensional connected closed unitary manifold $M^{2n}$
and a prime $p$. If there exists a partition
$\lambda=(1^{m_1(\lambda)}2^{m_2(\lambda)}\cdots n^{m_n(\lambda)})$
of weight $n$ such that the corresponding Chern number
$c_{\lambda}[M]$ is not divisible by $p$, then any $\mathbb{Z}_p$
action on $M^{2n}$ has at least $m(\lambda)+1$ fixed points.
\end{theorem}

\begin{remark}
Here we still don't need the assumption that the $\mathbb{Z}_p$
action has isolated fixed points and the reason is the same as that
of Remark \ref{remark}.
\end{remark}

\begin{corollary}\label{coro2}
Given a $2n$-dimensional connected closed unitary manifold $M^{2n}$
and a prime $p$. Let
$$m:=\textrm{max}\{m(\lambda)~|~c_{\lambda}[M]~\textrm{is not divisible by}~p\}.$$
Then any $\mathbb{Z}_p$ action on $M^{2n}$ has at least $m+1$ fixed
points.
\end{corollary}

The following result is an obstruction to the existence of a
$\mathbb{Z}_p$ action on $M^{2n}$ with isolated fixed points, which
is a direct application of Corollary \ref{coro}.

\begin{theorem}\label{end theorem}
Given a $2n$-dimensional connected closed unitary manifold $M^{2n}$
and a prime $p$. Let
$$m:=\textrm{max}\{m(\lambda)~|~c_{\lambda}[M]~\textrm{is not divisible by}~p\}.$$
If there exists a $\mathbb{Z}_p$ action on $M^{2n}$ with isolated
fixed points, then
$$p\geq m+1.$$
\end{theorem}

\begin{remark}
\begin{enumerate}
\item
As mentioned in (\cite{Ko}, p. 284), there is a similar formula like
Proposition \ref{Kos} in the smooth case for $\mathbb{Z}_p$ actions
so long as $p$ is an odd prime. So we have similar results like
Theorem \ref{p-theorem}, Corollary \ref{coro2} and Theorem \ref{end
theorem} for smooth $\mathbb{Z}_p$ action on $4n$-dimensional
smooth, closed and oriented manifolds in terms of Pontrjagin numbers
so long as $p$ is an odd prime.

\item
If $p$ is an odd prime and $\mathbb{Z}_p$ acts smoothly on a closed,
oriented, smooth manifold $N^{2n}$ with isolated fixed points such
that the fixed points satisfy some additional assumption, then Ewing
and Kosniowski also gives a lower bound of the number of these fixed
points in terms of $n$ and $p$ (\cite{EK}, p. 295).
\end{enumerate}
\end{remark}

\end{document}